\documentclass{amsart}

\usepackage{amsfonts}
\usepackage{amssymb}
\usepackage{amsthm}
\usepackage{eucal}
\usepackage{tikz}

\theoremstyle{plain}
\newtheorem{theorem}{Theorem}[section]
\newtheorem{lemma}[theorem]{Lemma}

\newtheorem{corollary}[theorem]{Corollary}

\theoremstyle{definition}

\newtheorem*{remark}{Remark}
\newtheorem*{acknowledgment}{Acknowledgment}

\newcommand{\m}{\mathfrak{m}}
\newcommand{\M}{\mathfrak{M}}
\newcommand{\J}{\mathfrak{J}}

\newcommand{\Z}{\mathbb{Z}}
\newcommand{\1}{\textbf{1}}
\newcommand{\e}{\textbf{e}}
\newcommand{\Supp}{\mathrm{Supp}}
\newcommand{\Aut}{\mathrm{Aut}}
\newcommand{\Ker}{\mathrm{Ker}}
\newcommand{\Inn}{\mathrm{Inn}}
\newcommand{\Mat}{\mathrm{Mat}}
\newcommand{\ve}{\varepsilon}
\newcommand{\F}{\mathbb{F}}
\newcommand{\wh}{\widehat}
\newcommand{\p}{\mathfrak{p}}

\begin{document}
\title[Jacobson graph]{Isomorphisms between Jacobson graphs}
\author[A. Azimi, A. Erfanian and M. Farrokhi D. G.]{A. Azimi, A. Erfanian and M. Farrokhi D. G.}

\subjclass[2000]{Primary 05C25, 05C60, Secondary 16P10, 13H99, 16N20.}
\keywords{Jacobson graph, isomorphism, automorphism, wreath product}

\address{Department of Pure Mathematics, Ferdowsi University of Mashhad, Mashhad, Iran}
\email{ali.azimi61@gmail.com}
\address{Department of Pure Mathematics, Ferdowsi University of Mashhad, Mashhad, Iran}
\email{erfanian@math.um.ac.ir}
\address{Department of Pure Mathematics, Ferdowsi University of Mashhad, Mashhad, Iran}
\email{m.farrokhi.d.g@gmail.com}

\begin{abstract}
Let $R$ be a commutative ring with a non-zero identity and $\J_R$ be its Jacobson graph. We show that if $R$ and $R'$ are finite commutative rings, then $\J_R\cong\J_{R'}$ if and only if $|J(R)|=|J(R')|$ and $R/J(R)\cong R'/J(R')$. Also, for a Jacobson graph $\J_R$, we obtain the structure of group $\Aut(\J_R)$ of all automorphisms of $\J_R$ and prove that under some conditions two semi-simple rings $R$ and $R'$ are isomorphic if and only if $\Aut(\J_R)\cong\Aut(\J_{R'})$.
\end{abstract}
\maketitle
\section{introduction}
Let $\Gamma$ and $\Gamma'$ be two simple graphs. An \textit{isomorphism} between $\Gamma$ and $\Gamma'$ is a bijection $\alpha:V(\Gamma)\longrightarrow V(\Gamma')$ such that
\[\{u,v\}\in E(\Gamma)\Longleftrightarrow\{\alpha(u),\alpha(v)\}\in E(\Gamma').\]
Two graphs $\Gamma$ and $\Gamma'$ are isomorphic whenever there is an isomorphism between them, and the isomorphism is denoted by $\Gamma\cong \Gamma'$. If $\Gamma'=\Gamma$, then an isomorphism between $\Gamma$ and $\Gamma'$ is said to be an \textit{automorphism} of $\Gamma$. The set of all automorphisms of $\Gamma$ is denoted by $\Aut(\Gamma)$.

Let $R$ and $R'$ be commutative rings with non-zero identities. The zero-divisor graph $\Gamma(R)$ of $R$ is a graph whose vertices are non-zero zero-divisors and two distinct vertices $x$ and $y$ are adjacent if and only if $xy=0$. Anderson, Frazier, Lauve and Livingston \cite{dfa-af-al-psl} proved that if $R$ and $R'$ are finite reduced rings which are not fields, then $\Gamma(R)\cong\Gamma(R')$ if and only if $R\cong R'$. Later, Akbari and Mohammadian \cite{sa-am} generalized this result and proved that if $R$ is a finite reduced ring which is not isomorphic to $\Z_2\oplus\Z_2$ and $\Z_6$ and $R'$ is a ring such that $\Gamma(R)\cong \Gamma(R')$, then $R\cong R'$.

Anderson and Livingston \cite{dfa-psl} have shown that $\Aut(\Gamma(\Z_n))$ is a direct product of some symmetric groups when $n\geq4$ is a non-prime integer. In the case of non-commutative rings, Han \cite{jh} shows that $\Aut(\Gamma(R))\cong S_{p+1}$ is the symmetric group of degree $p+1$ when $R=\Mat_2(\Z_p)$ ($p$ prime), and Park and Han \cite{sp-jh} generalized Han's result and proved that $\Aut(\Gamma(R))\cong S_{|F|+1}$ when $R=\Mat_2(F)$ with $F$ a finite field.
  
Let $R$ be a commutative ring with non-zero identity. The \textit{Jacobson graph} of $R$, denoted by $\J_R$, is a graph whose vertex set is $R\setminus J(R)$ such that two distinct vertices $x,y\in V(\J_R)$ are adjacent whenever $1-xy\not\in U(R)$, in which $U(R)$ is the group of units of $R$. It is easy to see that the subgraph of $\J_R$ induced by all vertices $x$ such that $1-x^2$ is a unit, is the co-normal product of some graphs, each of which is a union of isomorphic complete bipartite graphs. Recall that the Jacobson radical $J(R)$ of $R$ is the intersection of all maximal ideals of $R$ and it has the property that $1-xr\in U(R)$ whenever $x\in J(R)$ and $r\in R$. A ring with trivial Jacobson radical is called a \textit{semi-simple} ring.

The Jacobson graphs first introduced in \cite{aa-ae-mfdg} by the authors and several graph theoretical properties as well as numerical invariants of such graphs is discovered. In this paper, we shall concentrate on isomorphisms between Jacobson graphs. Also, we shall determine the structure of automorphism group of Jacobson graphs and give some conditions under which two rings are isomorphic whenever the automorphism group of their Jacobson graphs are isomorphic. In this paper, all rings will be commutative rings with non-zero identity.

Let $\Gamma$ be a graph. In what follows, $N_{\Gamma}(x)$ and $N_{\Gamma}[x]$ or simply $N(x)$ and $N[x]$ stand for the open and closed neighborhood of $x$, respectively, for every vertex $x$ of $\Gamma$. Also, $\deg_\Gamma(x)$ or simply $\deg(x)$ denotes the degree of $x$ and $\delta(\Gamma)$ denotes the minimum vertex degree of $\Gamma$.
\section{Isomorphic Jacobson graphs}
We begin with recalling some results from ring theory. Let $R$ be a finite ring. Then, by \cite[Theorem VI.2]{brm}, $R=R_1\oplus\cdots\oplus R_n$, where $(R_i,\m_i)$ are local rings with associated fields $F_i$. Moreover,
\[J(R)=J(R_1)\oplus\cdots\oplus J(R_n)\]
and
\[U(R)=U(R_1)\oplus\cdots\oplus U(R_n).\]
Clearly, 
\[|U(R)|=\prod_{i=1}^n|U(R_i)|=|J(R)|\prod_{i=1}^n\left(|F_i|-1\right).\]
In what follows, $\1$ stands for the element $(1,\ldots,1)$ with all entries equal to $1$ and $\e_i$ is the element with $1$ on its $i$th entry and $0$ elsewhere, for $i=1,\ldots,n$. Also, $e(R)$ and $o(R)$ denote the number of $F_i$ such that $|F_i|$ is even and odd, respectively. Note that, by \cite[Theorem 5.8]{aa-ae-mfdg}, for a unit element $x\in R$, $\deg_{\J_R}(x)=\left|\bigcup_{i=1}^n\M_i\right|-\ve_x$, where $\ve_x=0$ if $1-x^2\in U(R)$, $\ve_x=1$ if $1-x^2\notin U(R)$ and $\M_1,\ldots,\M_n$ are the maximal ideals of $R$. In particular, 
\[\deg_{\J_R}(\1)=\left|\bigcup_{i=1}^n\M_i\right|-1=|R|-|J(R)|\prod_{i=1}^n\left(|F_i|-1\right)-1.\]
Now, we shall study isomorphisms between Jacobson graphs. It is easy to see that $\Z_4$ and $\Z_2[x]/(x^2)$ have the same Jacobson graphs while they are non-isomorphic rings. The following theorem gives necessary and sufficient conditions under which two rings have isomorphic same Jacobson graphs.
\begin{theorem}
Let $R$ and $R'$ be finite rings. Then  $\J_R\cong\J_{R'}$  if and only if $|J(R)|=|J(R')|$ and $R/J(R)\cong R'/J(R')$.
\end{theorem}
\begin{proof}
Let $\varphi:\J_R\longrightarrow\J_{R'}$ be a graph isomorphism. Since $R$ and $R'$ are finite rings, we may write $R=R_1\oplus\cdots\oplus R_m$ and $R'=R'_1\oplus\cdots\oplus R'_n$, where $R_i$ and $R'_j$ are local rings with associated fields $F_i$ and $F'_j$, respectively. We further assume that $|F_1|\leq\cdots\leq|F_m|$ and $|F'_1|\leq\cdots\leq|F'_n|$. It is easy to see that the unit vertices have largest degrees among all other vertices and that there is two possible values for the degree of such elements (see \cite[Theorem 5.8]{aa-ae-mfdg}). Hence $\varphi$ maps the units of $R$ onto the units of $R'$ and consequently $|U(R)|=|U(R')|$. In particular, we have $\deg(\1_R)=\deg(\1_{R'})$. Therefore,
\[|J(R)|\prod_{i=1}^m\left(|F_i|-1\right)=|J(R')|\prod_{j=1}^n\left(|F'_j|-1\right)\]
and
\[|R|-|J(R)|\prod_{i=1}^m\left(|F_i|-1\right)-1=|R'|-|J(R')|\prod_{j=1}^n\left(|F'_j|-1\right)-1,\]
from which it follows that $|R|=|R'|$ and consequently $|J(R)|=|J(R')|$ as the graphs have the same number of vertices.

If $|F_m|\leq3$ and $|F'_n|\leq3$, then 
\[2^{e(R)}3^{o(R)}=\left|\frac{R}{J(R)}\right|=\left|\frac{R'}{J(R')}\right|=2^{e(R')}3^{o(R')},\]
which implies that $R/J(R)\cong R'/J(R')$. Now, suppose that $|F_m|>3$. It is easy to see that
\[\delta(\J_R)=\deg_{\J_R}(\e_m)=\frac{|R|}{|F_m|}-1\]
and
\[\delta(\J_{R'})=\deg_{\J_{R'}}(\e'_n)=\frac{|R'|}{|F'_n|}-1.\]
Since $\delta(\J_R)=\delta(\J_{R'})$ we obtain $|F_m|=|F'_n|$ and hence $F_m\cong F'_n$. Let $S=R_1\oplus\cdots\oplus R_{m-1}$ and $S'=R'_1\oplus\cdots\oplus R'_{n-1}$. By \cite[Lemma 2.1(2) and Theorem 2.2]{aa-ae-mfdg}, there exist elements $a_m\in R_m\setminus J(R_m)$ and $b_n\in R'_n\setminus J(R'_n)$ such that $1-a_m^2\in U(R_m)$ and $1-b_n^2\in U(R'_n)$. Without loss of generality, we may assume that $\varphi(a_m\e_m)=b_n\e'_n$. Since $\varphi(N_{\J_R}(a_m\e_m))=N_{\J_{R'}}(b_n\e'_n)$, the map
\[\begin{array}{rcl}\phi:\J_S&\longrightarrow&\J_{S'}\\x&\longmapsto&\varphi(x+a_m^{-1}\e_m)\cdot(\1_{R'}-\e'_n)\end{array}\]
is an isomorphism between $\J_S$ and $\J_{S'}$. Using induction we may assume that $S/J(S)\cong S'/J(S')$, from which it follows that $R/J(R)\cong R'/J(R')$.
The converse is clear.
\end{proof}
\begin{corollary}
Let $R$ and $R'$ be finite semi-simple rings. Then $\J_R\cong\J_{R'}$ if and only if $R\cong R'$.  
\end{corollary}
\begin{remark}
As infinite fields with the same characteristic and cardinality have isomorphic Jacobson graphs, finiteness is a necessary condition in the above theorem.
\end{remark}
\section{Automorphisms of Jacobson graphs}
In what follows, the \textit{support} of an element $x=(x_1,\ldots,x_n)$ in a semi-simple ring $R=F_1\oplus\cdots\oplus F_n$, where $F_i$ are fields, is the set of all indices $i$ such that $x_i\neq0$ and it is denoted by $\Supp(x)$. Our results are based on the following two lemmas.
\begin{lemma}\label{openneighborhood}
Let $R$ be a finite semi-simple ring such that $R\not\cong\Z_2\oplus\Z_2$ and $R$ is not a field. If $x,y\in V(\J_R)$, then $N(x)=N(y)$ if and only if $x=y$. 
\end{lemma}
\begin{proof}
Let $R=F_1\oplus\cdots\oplus F_n$ be the decomposition of $R$ into finite fields $F_i$. Suppose that $x=(x_1,\ldots,x_n)$ and $y=(y_1,\ldots,y_n)$ are vertices of $\J_R$ such that $N(x)=N(y)$ but $x\neq y$. Then there exists $i\in\{1,\ldots,n\}$ such that $x_i\neq y_i$ and without loss of generality $x_i\neq0$. If $x_i\neq\pm1$, then $x_i^{-1}\e_i\in N(x)$ so that $y_i=x_i$, a contradiction. Hence $x_i=\pm1$. If $|\Supp(x)|\geq2$, then since $x_i\e_i\in N(x)$ we have that $y_i=x_i$, a contradiction. Thus $\Supp(x)=\{i\}$ and $x=\pm\e_i$. Similarly $y=\pm\e_j$. Now a simple verification shows that $N(x)=N(y)$ only if $R=\Z_2\oplus\Z_2$, which is a contradiction. The converse is obvious.
\end{proof}
\begin{lemma}\label{closedneighborhood}
Let $R$ be a finite semi-simple ring that is not a field. If $x,y\in V(\J_R)$, then $N[x]=N[y]$ if and only if $x=y$. 
\end{lemma}
\begin{proof}
Let $R=F_1\oplus\cdots\oplus F_n$ be the decomposition of $R$ into finite fields $F_i$. Suppose that $x=(x_1,\ldots,x_n)$ and $y=(y_1,\ldots,y_n)$ are vertices of $\J_R$ such that $N[x]=N[y]$ but $x\neq y$. Clearly, $x$ and $y$ are adjacent. If $\Supp(x)=\{i\}$ is a singleton, then $x=x_i\e_i$ and $y_i=x_i^{-1}$. If $y\neq x_i^{-1}\e_i$ and $j\in\Supp(y)\setminus\{i\}$, then $y_j^{-1}\e_j\in N[y]\setminus N[x]$, which is impossible. Thus $y=x_i^{-1}\e_i$. If $x_i\neq\pm1$ and $j\neq i$, then $x_i\e_i+\e_j\in N[y]\setminus N[x]$, which is impossible. Thus $x_i=\pm1$ and hence $x=y$, a contradiction. Thus $|\Supp(x)|>1$ and similarly $|\Supp(y)|>1$. Suppose that $x_i\neq y_i$ and without loss of generality $x_i\neq0$. Then $x_i^{-1}\e_i\in N(x)$, which implies that $x_i^{-1}\e_i\in N(y)$. Hence $y_i=x_i$, which is a contradiction. The converse is obvious.
\end{proof}

Let $R$ be a finite ring and $\theta:\Aut(\J_R)\longrightarrow\Aut(\J_{R/J(R)})$ be the map give by 
$\theta(\alpha)=\overline{\alpha}$ such that $\overline{\alpha}(x+J(R))=\alpha(x)+J(R)$ for every $x\in R\setminus J(R)$. With this definition, $\theta$ gives rise to an epimorphism from $\Aut(\J_R)$ onto $\Aut(\J_{R/J(R)})$. Therefore 
\[\frac{\Aut(\J_R)}{\Ker(\theta)}\cong\Aut\left(\J_{\frac{R}{J(R)}}\right).\]
Moreover, 
\[\Ker(\theta)=S_{|J(R)|}\times\cdots\times S_{|J(R)|}\]
is the direct product of $|R/J(R)|-1$ copies of the symmetric group on $|J(R)|$ letters. So, in the remainder of this section, we just consider the case of finite semi-simple rings.
\begin{theorem}\label{automorphism}
Let $R=F_1\oplus\cdots\oplus F_n$ be the decomposition of semi-simple ring $R$ into finite fields $F_i$  and $|F_1|\geq\cdots \geq |F_n|$. Also, let $A_i=\{j:|F_j|=|F_i|\}$, $\Pi\leq S_n$ be the set of all permutations preserving $A_i$'s, and $\Sigma_i\leq S_{|F_i|}$ be the set of all permutations fixing $0$ and preserving the sets $\{\pm1\}$ and $\{f_i,f_i^{-1}\}$ for all $f_i\in F_i\setminus\{0,\pm1\}$. Then
\[\Aut(\J_R)=\left\{\alpha:\J_R\longrightarrow\J_R:\alpha(\sum_{i=1}^{n}a_i\e_i)=\sum_{i=1}^{n}\sigma_i(a_i)\e_{\pi(i)},\pi\in\Pi,\sigma_i\in \Sigma_i\right\}.\]
\end{theorem}
\begin{proof}
Let $\Inn(\J_R)$ be the set of automorphisms in the theorem. We show that $\Aut(\J_R)=\Inn(\J_R)$. 
We proceed by induction on $n$. Clearly, the result holds for $n=1$. Hence we assume that $n>1$ and the result holds for $n-1$. Let $\alpha\in\Aut(\J_R)$, and $a,b\in F_1\setminus\{0\}$. Then $\deg(a\e_1),\deg(b\e_1)=|R|/|F_1|-1$ or $|R|/|F_1|$ according to the cases when $a,b=\pm1$ or not, respectively. On the other hand, for a vertex $x\in V(\J_R)$, $\deg(x)=|R|/|F_1|-1$ or $|R|/|F_1|$ only if $x=f_k\e_k$ for some $k\in A_1$. Hence $\alpha(a\e_1)=f_i\e_i$ and $\alpha(b\e_1)=f_j\e_j$ for some $i,j\in A_1$. If $i\neq j$, then $N_{\J_R}(f_i\e_i)\cap N_{\J_R}(f_j\e_j)\neq\emptyset$ while $N_{\J_R}(a\e_1)\cap N_{\J_R}(b\e_1)=\emptyset$, a contradiction. Thus $i=j$ and we may assume modulo $\Inn(\J_R)$ that $\alpha(f\e_1)=f\e_1$ for all $f\in F_1\setminus\{0\}$. In particular, $\alpha(N_{\J_R}(f\e_1))=N_{\J_R}(f\e_1)$ for all $f\in F_1\setminus\{0\}$. Let $S=F_2\oplus\cdots\oplus F_n$. Since $N_{\J_R}(f\e_1)=f^{-1}\e_1+S$ for all $f\in F_1\setminus\{0\}$, it follows that $\alpha(S)=S$. Let $\alpha|_S:\J(S)\longrightarrow \J(S)$ be the restriction of $\alpha$ on $S\setminus\{0\}$. By assumption, $\alpha|_S\in\Inn(\J_S)$. Hence we may further assume modulo $\Inn(\J_R)$ that $\alpha(x)=x$ for all $x\in S\setminus\{0\}$. Let $S_f=f\e_1+S$ for all $f\in F_1\setminus\{0\}$. Since $\alpha$ fixes $S_f$, for each $x\in S$, there exists $\alpha_f(x)\in S$ such that $\alpha(f\e_1+x)=f\e_1+\alpha_f(x)$.

Let $f\in F\setminus\{0,\pm1\}$. If $x,y\in S\setminus\{0\}$, then $x\sim y$ if and only if $f\e_1+x\sim f\e_1+y$ if and only if $f\e_1+\alpha_f(x)\sim f\e_1+\alpha_f(y)$ if and only if $\alpha_f(x)\sim\alpha_f(y)$. Hence $\alpha_f$ induces an automorphism on $\J_S$. On the other hand, if $x,y\in S\setminus\{0\}$, then $x\sim y$ if and only if $x\sim f\e_1+y$ if and only if $x\sim f\e_1+\alpha_f(y)$ if and only if $x\sim\alpha_f(y)$ or $x=\alpha_f(y)$. Now it is easy to see that, if $x\not\sim\alpha_f(x)$, then $\alpha_f(N_{\J_S}(x))=N_{\J_S}(x)$ and if $x\sim\alpha_f(x)$, then $\alpha_f(N_{\J_S}[x])=N_{\J_S}[x]$. Hence, by Lemmas \ref{openneighborhood} and \ref{closedneighborhood}, $\alpha_f(x)=x$ except probably when $\alpha_f(N_{\J_S}(x))=N_{\J_S}(x)$ and $S\cong\Z_2\oplus \Z_2$, or $S$ is field. If $S\cong\Z_2\oplus\Z_2$, then since $F+\e_i\subseteq N_{\J_R}[\e_i]$, it follows that $\alpha(F+\e_i)\subseteq (F+\e_i)\cup(F+\e_2+\e_3)$, for $i=2,3$. But $\deg_{\J_R}(x)<\deg_{\J_R}(y)$ for all $x\in F+\e_i$ and $y\in F+\e_2+\e_3$, which implies that $\alpha(F+\e_i)=F+\e_i$, for $i=2,3$. Moreover, $\alpha(F+\e_2+\e_3)=F+\e_2+\e_3$. Thus $\alpha$ stabilizes $N_{\J_R}[f\e_1+\e_2]$, $N_{\J_R}[f\e_1+\e_3]$ and $N_{\J_R}[f\e_1+\e_2+\e_3]$, from which by Lemma \ref{closedneighborhood}, it follows that $\alpha$ fixes all the elements $f\e_1+\e_2$, $f\e_1+\e_3$ and $f\e_1+\e_2+\e_3$, as required. Now, suppose that $S$ is a field and $\alpha_f(x)\neq x$. If $\alpha_f(N_{\J_S}[x])=N_{\J_S}[x]$, then $\alpha_f(x)=x^{-1}$. But then $\alpha_f(x^{-1})=x$ and since $x\sim f\e_1+x^{-1}$, we should have $x\sim f\e_1+\alpha_f(x^{-1})$, a contradiction. Thus $\alpha_f(N_{\J_S}(x))=N_{\J_S}(x)$, which implies that $x=\pm1$ and $\alpha_f(x)=-x$. Since $x\sim f\e_1+x$ it follows that $x\sim f\e_1+\alpha_f(x)$, which is a contradiction. Therefore, in all cases $\alpha_f(x)=x$ that is $\alpha(f\e_1+x)=f\e_1+x$ for all $x\in S$.

Finally, suppose that $f=\pm1$. Let
\[\overline{N}_{\J_S}(x)=\begin{cases}N_{\J_S}(x),&1-x^2\in U(S),\\\\N_{\J_S}[x],&1-x^2\notin U(S),\end{cases}\]
for each $x\in S\setminus\{0\}$. Then 
\[N_{\J_R}(x)=N_{\J_S}(x)\cup\bigcup_{f'\in F\setminus\{0\}}\left(f'\e_1+\overline{N}_{\J_S}(x)\right).\]
Since $\alpha(x)=x$, it follows that $\alpha(N_{\J_R}(x))=N_{\J_R}(x)$. Hence 
\begin{equation}\label{fixedcosets}
\alpha(f'\e_1+\overline{N}_{\J_S}(x))=f'\e_1+\overline{N}_{\J_S}(x)
\end{equation}
for each $f'\in F_1\setminus\{0\}$. Now, we have
\[N_{\J_R}[f\e_1+x]=(f\e_1+S)\cup\bigcup_{f'\in F_1\setminus\{f\}}\left(f'\e_1+\overline{N}_{\J_S}(x)\right)\]
so that, by using \eqref{fixedcosets}, $\alpha(N_{\J_R}[f\e_1+x])=N_{\J_R}[f\e_1+x]$. Thus $N_{\J_R}[\alpha(f\e_1+x)]=N_{\J_R}[f\e_1+x]$ and by Lemma \ref{closedneighborhood}, $\alpha(f\e_1+x)=f\e_1+x$. Therefore, $\alpha=I$ is the identity automorphism and the proof is complete.
\end{proof}

For each finite field $F$, we let $\ve_F=0$ if $|F|$ is even and $\ve_F=1$ if $|F|$ is odd. By \cite[Theorem 2.2]{aa-ae-mfdg}, $\J_F$ is a union of $\ve_F+1$  isolated vertices together with $(|F|-2-\ve_F)/2$ disjoint edges. Hence, we have the following result.
\begin{corollary}\label{automorphismstructure}
Let $R=F_1^{m_1}\oplus\cdots\oplus F_k^{m_k}$ be the decomposition of the semi-simple ring $R$ into finite fields $F_i$, where $F_1,\ldots,F_k$ are distinct fields. Then
\[\Aut(\J_R)=\Aut(\J_{F_1})\wr S_{m_1}\times\cdots\times\Aut(\J_{F_k})\wr S_{m_k},\]
in which 
\[\Aut(\J_{F_i})=\left(\Z_2\wr S_{\frac{|F_i|-2-\ve_{F_i}}{2}}\right)\times\Z_{\ve_{F_i}+1}\]
if $|F_i|\geq4$, and $\Aut(\Z_3)=\Z_2$ and $\Aut(\Z_2)=1$ is the trivial group.
\end{corollary}
\begin{remark}
Let $R=\Z_2\oplus\cdots\oplus\Z_2$ ($n$ times). Then, by Corollary \ref{automorphismstructure}, $\Aut(\J_R)\cong S_n$. Note that $\J_R$ is the complement of the zero-divisor graph of $R$ and it is known, say from \cite{am}, that $\Aut(\J_R)\cong S_n$. Hence $\J_R$ is an example of graphs with symmetric automorphism group, in which the number of vertices exceeds that of \cite{mwl}.
\end{remark}
\begin{remark}
Let $R=F\oplus\cdots\oplus F$ ($n$ times), where $F$ is a field of order $3$ or $4$. Then, by Corollary \ref{automorphismstructure}, $\Aut(\J_R)\cong \Z_2\wr S_n$. Hence, by \cite{fh}, $\Aut(\J_R)\cong\Aut(Q_n)$, where $Q_n$ denotes the hypercube of dimension $n$.
\end{remark}

In the remainder of this section, we shall seek for conditions under which two finite rings are isomorphic if and only if the automorphism group of their Jacobson graphs are isomorphic.
\begin{lemma}\label{wreathproductdecomposition}
Let $G$ be a finite group and $W=G\wr S_n$. Then $W$ is decomposable if and only if $G$ has a non-trivial central  direct factor $H$ such that $(|H|,n)=1$, in which case $H$ is a direct factor of $W$.
\end{lemma}
\begin{proof}
Suppose $W$ is decomposable. By \cite[Theorem 2.8]{mrd-taf}, $G$ has a non-trivial Abelian direct factor $H$ such that its elements each of which has unique $n$th root. Suppose on the contrary that $\gcd(|H|,n)\neq1$ and $p$ be a common prime divisor of $|H|$ and $n$. Let $x$ be a $p$-element of largest order $p^k$ in $H$ and $y$ be the unique $n$th root of $x$ in $H$. Then $p\big||y|$, which implies that $|y^n|<p^k$, a contradiction. The converse follows from \cite[Theorem 2.8]{mrd-taf}.
\end{proof}

In the remaining of this section, $\F_m$ stands for the finite field of order $m$ and $\wh{m}$ denotes the number $(m-2-\ve_m)/2$ for each prime power $m$, in which $\ve_m=m-2[m/2]$ denotes the parity of $m$. If $p$ is a prime, then $e_p(n)$ is the exponent of $p$ in $n$, that is $p^{e_p(n)}|n$ but $p^{e_p(n)+1}\nmid n$. Bertrand's postulate states that there always exists a prime $p$ such that $n<p<2n-2$ for all $n>3$. Also, a weaker formulation of Bertrand's postulate states that there always exists a prime $p$ such that $n<p<2n$ for all $n>1$. Hence, if $\p_n$ denotes the largest prime not exceeding $n$, then $e_{\p_n}(n!)=1$. Also, it is known that the following equality holds for each natural number $n$ and prime $p$,
\[e_p(n!)=\frac{n-S_p(n)}{p-1},\]
where $S_p(n)$ denotes the sum of digits of $n$ when $n$ is written in base $p$.

By Lemma \ref{wreathproductdecomposition}, we have
\[\Aut(\J_{\F_m})\cong\Z_{\ve_m+1}\times\Z_{\ve_{\wh{m}}+1}\times H_m,\]
where $|H_m|=2^{\wh{m}-\ve_{\wh{m}}}\wh{m}!$, and
\[\Aut(\J_{\F_m})\wr S_n\cong\begin{cases}S_n,&m=2,\\\Z_{\ve_n+1}\times K_n,&m=3,4,\\\Z_{\ve_{\wh{m}}+1}^{\ve_n}\times\Z_{\ve_m+1}^{\ve_n}\times L_{m,n},&m\geq5,\end{cases}\]
where $|K_n|=2^{n-\ve_n}n!$ and $|L_{m,n}|=2^{-\ve_n(\ve_{\wh{m}}+\ve_m)}\left(2^{\wh{m}+\ve_m}\wh{m}!\right)^nn!$.
\begin{lemma}\label{isomorphicindecomposablefactors}
Let $m,m',n,n',k$ be two natural numbers greater than one. Then
\begin{itemize}
\item[(1)]$S_k\cong H_m$, if and only if $(k,\wh{m})=(4,3)$,
\item[(2)]$S_k\cong K_n$, if and only if $(k,n)=(4,3)$,
\item[(3)]$S_k\not\cong L_{m,n}$,
\item[(4)]$H_m\cong H_{m'}$, if and only if $\wh{m}=\wh{m}'$,
\item[(5)]$H_m\cong K_{n'}$, if and only if $\wh{m}=n'$,
\item[(6)]$H_m\not\cong L_{m',n'}$,
\item[(7)]$K_n\cong K_{n'}$, if and only if $n=n'$,
\item[(8)]$K_n\not\cong L_{m',n'}$,
\item[(9)]$L_{m,n}\cong L_{m',n'}$, if and only if $(m,n)=(m',n')$.
\end{itemize}
\end{lemma}
\begin{proof}
(1) If $S_k\cong H_m$, then $k!=2^{\wh{m}-\ve_{\wh{m}}}\wh{m}!$. Then $k>\wh{m}$, which implies that $\wh{m}=k-1$. Thus $k=2^{k-1-\ve_{k-1}}=2^{k-2}$, which holds only if $k=4$. On the other hand, $\Z_2\wr S_3\cong\Z_2\times S_4$ and the result follows.

(2) It is similar to case (1).

(3) If $S_k\cong L_{m,n}$, then $k!=2^l\wh{m}!^nn!$ for $l=(\wh{m}+\ve_m)n-\ve_n(\ve_{\wh{m}}+\ve_m)$, which implies that $k>n$. Since $e_{\p_k}(k!)=1$, it follows that $e_{\p_k}(\wh{m})=0$ and $e_{\p_k}(n!)=1$. Hence $\wh{m}<\p_k\leq n<k\leq 2n$. If $\wh{m}>2$, then since $e_{\p_{\wh{m}}}(k!)=e_{\p_{\wh{m}}}(2^l\wh{m}!^nn!)$, we have
\[n=e_{\p_{\wh{m}}}(k!)-e_{\p_{\wh{m}}}(n!)=\frac{k-S_{\p_{\wh{m}}}(k)}{\p_{\wh{m}}-1}-\frac{n-S_{\p_{\wh{m}}}(n)}{\p_{\wh{m}}-1}.\]
Thus 
\[(\p_{\wh{m}}-1)n=k-n+S_{\p_{\wh{m}}}(n)-S_{\p_{\wh{m}}}(k)<n+S_{\p_{\wh{m}}}(n),\]
which implies that $n\leq(\p_{\wh{m}}-2)n<S_{\p_{\wh{m}}}(n)$, a contradiction. If $\wh{m}=2$, then $k!=2^{(3+\ve_m)n-\ve_m\ve_n}n!$. Hence $n=k-1$ and $k=2^{(3+\ve_m)k-3-2\ve_m}\geq2^{3k-5}$, which is impossible since $k>2$. Also, if $\wh{m}=1$, then $m=5$ and $k!=2^{2n-2\ve_n}n!$. Hence $n=k-1$ and $k=2^{2k-4}$, which is impossible.

(4) If $H_m\cong H_{m'}$, then $2^{\wh{m}-\ve_{\wh{m}}}\wh{m}!=2^{\wh{m}'-\ve_{\wh{m}'}}\wh{m}'!$. Suppose on the contrary that $\wh{m}\neq\wh{m}'$ and $\wh{m}>\wh{m}'$, say. Then $\wh{m}'=\wh{m}-1$ and we have
\[\wh{m}=2^{(\wh{m}-1-\ve_{\wh{m}-1})-(\wh{m}-\ve_{\wh{m}})}=2^{-2},\]
which is not an integer, a contradiction.

(5) It is similar to case (4).

(6) If $H_m\cong L_{m',n'}$, then it is easy to see that $\wh{m}>n'$ and the same as in (3) we can show that $\wh{m}'\leq2$. If $\wh{m}'=2$, then $2^{\wh{m}-\ve_{\wh{m}}}\wh{m}!=2^{(3+\ve_{m'})n'-\ve_{m'}\ve_{n'}}n'!$. Hence $n'=\wh{m}-1$ and 
\begin{align*}
\wh{m}&=2^{\left((3+\ve_{m'})(\wh{m}-1)-\ve_{m'}\ve_{\wh{m}-1}\right)-(\wh{m}-\ve_{\wh{m}})}\\
&=2^{(2+\ve_{m'})\wh{m}-3-2\ve_{m'}}\\
&\geq2^{2\wh{m}-5},
\end{align*}
which is impossible since $\wh{m}>2$. Also, if $\wh{m}'=1$, then $m'=5$ and $2^{\wh{m}-\ve_{\wh{m}}}\wh{m}!=2^{2n'-2\ve_{n'}}n'!$. Hence $n'=\wh{m}-1$ and
\[\wh{m}=2^{\left(2(\wh{m}-1)-2\ve_{\wh{m}-1}\right)-(\wh{m}-\ve_{\wh{m}})}=2^{\wh{m}-4},\]
which is impossible.

(7) It is similar to case (4).

(8) It is similar to case (6).

(9) If $L_{m,n}\cong L_{m',n'}$, then $2^l\wh{m}!^nn!=2^{l'}\wh{m}'!^{n'}n'!$, where $l=(\wh{m}+\ve_m)n-\ve_n(\ve_{\wh{m}}+\ve_m)$ and $l'=(\wh{m}'+\ve_{m'})n'-\ve_{n'}(\ve_{\wh{m}'}+\ve_{m'})$. If $n=n'$, then \[2^{(\wh{m}+\ve_m)n-\ve_n(\ve_{\wh{m}}+\ve_m)}\wh{m}!^n=2^{(\wh{m}'+\ve_{m'})n-\ve_n(\ve_{\wh{m}'}+\ve_{m'})}\wh{m}'!^n.\]
If $\wh{m}\neq\wh{m}'$ and $\wh{m}>\wh{m}'$, say, then $\wh{m}'=\wh{m}-1$ and we have
\begin{align*}
\wh{m}^n&=2^{\left((\wh{m}-1+\ve_{m'})n-\ve_n(\ve_{\wh{m}-1}+\ve_{m'})\right)-\left((\wh{m}+\ve_m)n-\ve_n(\ve_{\wh{m}}+\ve_m)\right)}\\
&=2^{(\ve_{m'}-\ve_m-1)n-\ve_n(1+\ve_{m'}-\ve_m)}\leq1,
\end{align*}
which is a contradiction. Thus $\wh{m}=\wh{m}'$ and consequently $m=m'$. Thus we may assume that $n\neq n'$ and without loss of generality $n<n'$. First suppose that $\wh{m}>2$. Since $e_{\p_{\wh{m}}}(2^l\wh{m}!^nn!)=e_{\p_{\wh{m}}}(2^{l'}\wh{m}'!^{n'}n'!)$, we obtain
\[n=n'e_{\p_{\wh{m}}}(\wh{m}'!)+e_{\p_{\wh{m}}}(n'!)-e_{\p_{\wh{m}}}(n!)\geq n'e_{\p_{\wh{m}}}(\wh{m}'!).\]
Thus $e_{\p_{\wh{m}}}(\wh{m}'!)\leq n/n'<1$ and consequently $e_{\p_{\wh{m}}}(\wh{m}'!)=0$. Hence $\wh{m}'<\p_{\wh{m}}\leq\wh{m}$. Now, we have
\[n=e_{\p_{\wh{m}}}(n'!)-e_{\p_{\wh{m}}}(n!)=\frac{n'-S_{\p_{\wh{m}}}(n')}{\p_{\wh{m}}-1}-\frac{n-S_{\p_{\wh{m}}}(n)}{\p_{\wh{m}}-1}.\]
Thus 
\[n'=(\p_{\wh{m}}-1)n+S_{\p_{\wh{m}}}(n')+n-S_{\p_{\wh{m}}}(n)>(\p_{\wh{m}}-1)n\geq2n.\]

On the other hand, $e_{\p_{n'}}(2^l\wh{m}!^nn!)=e_{\p_{n'}}(2^{l'}\wh{m}'!^{n'}n'!)$ and by invoking Bertrand's postulate, if follows that $ne_{\p_{n'}}(\wh{m}!)=n'e_{\p_{n'}}(\wh{m}'!)+1$. Thus $e_{\p_{n'}}(\wh{m}!)\neq0$ and consequently $e_{\p_{n'}}(\wh{m}'!)\neq0$. Also, by Bertrand's postulate, we have	
\[\frac{\wh{m}-S_{\p_{n'}}(\wh{m})}{\wh{m}'-S_{\p_{n'}}(\wh{m}')}=\frac{e_{\p_{n'}}(\wh{m}!)}{e_{\p_{n'}}(\wh{m}'!)}>\frac{n'}{n}>\p_{\wh{m}}-1\geq\frac{\wh{m}}{2}.\]
Thus
\[\wh{m}'-S_{\p_{n'}}(\wh{m}')<\frac{2(\wh{m}-S_{\p_{n'}}(\wh{m}))}{\wh{m}}=2-2\frac{S_{\p_{n'}}(\wh{m})}{\wh{m}}<2.\]
Hence $\wh{m}'-S_{\p_{n'}}(\wh{m}')\leq1$. Since $\wh{m}'-S_{\p_{n'}}(\wh{m}')=(\p_{n'}-1)e_{\p_{n'}}(\wh{m}'!)$, it follows that $e_{\p_{n'}}(\wh{m}'!)=0$, which is a contradiction. 

Thus $\wh{m}\leq2$. If $\wh{m}=2$, then
\[2^{(3+\ve_m)n-\ve_m\ve_n}n!=2^{(\wh{m}'+\ve_{m'})n'-\ve_{n'}(\ve_{\wh{m}'}+\ve_{m'})}\wh{m}'!^{n'}n'!.\]
Hence $n=n'-1$ and $\wh{m}'\leq2$. If $\wh{m}'=2$, then $n'=2^{(\ve_m-\ve_{m'})n'-3-2\ve_m}$, which implies that $\ve_{m}=1$ and $\ve_{m'}=0$. Hence $n'=2^{n'-5}$, which holds only if $n'=8$ and consequently $n=7$. Then 
\[(\Z_2\wr S_2\times\Z_2)\wr S_7\cong\Z_2\times(\Z_2\wr S_2)\wr S_8,\]
which is impossible by Jordan-H\"{o}lder's theorem for the simple sections of the left hand side group are $\Z_2$ and $A_7$ while the simple sections of the right hand side group are $\Z_2$ and $A_8$. Also, if $\wh{m}'=1$, then $n'=2^{(2+\ve_m-\ve_{m'})n'-3-2\ve_m}\geq2^{2n'-5}$, which is impossible.

Finally, suppose that $\wh{m}=1$. Then $m=5$ and we have
\[2^{2n-2\ve_n}n!=2^{(\wh{m}'+\ve_{m'})n'-\ve_{n'}(\ve_{\wh{m}'}+\ve_{m'})}\wh{m}'!^{n'}n'!.\]
Hence $n=n'-1$ and $\wh{m}'=1$. Thus $m'=5$ so that $n'=2^{-4}$, which is not an integer. The proof is complete.
\end{proof}
\begin{theorem}
Let $R=F_1^{m_1}\oplus\cdots\oplus F_s^{m_s}$ and $R'={F'}_1^{n_1}\oplus\cdots\oplus {F'}_t^{n_t}$ be two finite semi-simple rings. 
\begin{itemize}
\item[(1)]If $m_i,n_j>1$ for $i=1,\ldots,s$ and $j=1,\ldots,t$, and either the number of fields of order $2$ in $R$ and $R'$ is not equal to $4$ or the number of fields of order $3$ and $4$ in $R$ and $R'$ is not equal to $3$, then $R\cong R'$ if and only if $\Aut(\J_R)\cong\Aut(\J_{R'})$.
\item[(2)]If all the fields $F_i$ and $F_j$ have the same parity of order different from $3$ and $4$, and either the number of fields of order $2$ in $R$ and $R'$ is not equal to $4$ or the number of fields of order $8$ and $9$ in $R$ and $R'$ is greater than $1$, then $R\cong R'$ if and only if $\Aut(\J_R)\cong\Aut(\J_{R'})$.
\end{itemize}
\end{theorem}
\begin{proof}
We know, by Krull, Remak and Schmidt Theorem \cite[Theorem 3.3.8]{djsr}, that every finite group can be expressed, up to isomorphism, as a direct product of some indecomposable groups. Now, the result follows by using Corollary \ref{automorphismstructure} and Lemmas \ref{wreathproductdecomposition} and \ref{isomorphicindecomposablefactors}.
\end{proof}
\begin{acknowledgment}
The authors would like to thank Prof. Peter M. Neumann for pointing out that the groups in Lemma 3.6 have different orders except for few possibilities.
\end{acknowledgment}

\end{document}